\documentclass[11pt]{amsart}

\usepackage[margin=1.05in]{geometry}
\usepackage{amsmath,amssymb,amsthm}
\usepackage{graphicx}
\usepackage{tikz-cd}
\usepackage{enumitem}
\usepackage{microtype}
\usepackage[hidelinks]{hyperref}

\newtheorem{theorem}{Theorem}[section]
\newtheorem{conjecture}[theorem]{Conjecture}
\newtheorem{proposition}[theorem]{Proposition}
\newtheorem{lemma}[theorem]{Lemma}
\newtheorem{corollary}[theorem]{Corollary}
\theoremstyle{remark}
\newtheorem{remark}[theorem]{Remark}

\newcommand{\Z}{\mathbb{Z}}
\newcommand{\R}{\mathbb{R}}
\newcommand{\Ftwo}{\mathbb{F}_2}
\newcommand{\Fol}{\mathcal{F}}
\newcommand{\Fhat}{\widehat{F}}
\newcommand{\psihat}{\widehat{\psi}}
\newcommand{\Folhat}{\widehat{\mathcal{F}}}
\newcommand{\Sigmahat}{\widehat{\Sigma}}
\DeclareMathOperator{\rank}{rank}
\DeclareMathOperator{\Hom}{Hom}
\DeclareMathOperator{\Int}{int}

\title[Lens Space Surgeries and the Bleiler--Litherland Conjecture]
{Lens Space Surgeries and the Bleiler--Litherland Conjecture}
\author{Qilong Guo}
\address{College of Science, China University of Petroleum-Beijing,
Beijing 102249, China}
\email{guoqilong1984@hotmail.com}
\date{}
\hypersetup{pdftitle={Lens Space Surgeries and the Bleiler--Litherland Conjecture},
  pdfauthor={Qilong Guo}}

\begin{document}

\begin{abstract}
We prove the Bleiler--Litherland conjecture: every lens space obtained by a
nontrivial Dehn surgery on a hyperbolic knot in $S^3$ has order at least $18$.
The key ingredient is a spectral obstruction: if a hyperbolic knot $K$ of
genus $g$ admits a lens space surgery with slope $\pm(4g-2)$, then
$\Delta_K(t)$ has a real root outside the unit circle. The proof combines
Floer-theoretic restrictions on lens space surgeries, Gabai's degeneracy-slope
bound and Gabai--Oertel's persistence theorem for essential laminations,
Ni's fixed-point theorem for monodromy, and a mod-$2$ orientability criterion,
together with the relation between homological monodromy and the Alexander
polynomial. As a further application of this spectral obstruction, we obtain
characterizing-slope results for torus knots.
\end{abstract}

\maketitle

\section{Introduction}

For a knot $K\subset S^3$ and coprime integers $p,q$ with $q\geq 0$, let
$S^3_{p/q}(K)$ denote the $3$-manifold obtained by $p/q$-surgery on $K$, with
respect to the meridian--longitude basis. When $q=0$, take $p=1$, so $1/0$
is the meridional slope. A surgery is called nontrivial if its slope is
nonmeridional. When $q>0$, the surgery is positive if $p>0$ and integer if
$q=1$. For coprime integers $r,s$ with $r>0$, let $L(r,s)=S^3_{r/s}(U)$,
where $U$ is the unknot. We call $r$ the order of $L(r,s)$. Thus the order
of a lens space is the cardinality of its first homology group. A lens space
surgery is a nontrivial surgery whose resulting manifold is a lens space.

Moser classified lens space surgeries on torus knots \cite{Moser}.
Fintushel and Stern constructed hyperbolic knots with lens space surgeries
\cite{FintushelStern}. Their examples led Bleiler and Litherland to the
following conjecture \cite{BleilerLitherland}.

\begin{conjecture}[Bleiler--Litherland]\label{conj:BL}
Every lens space obtained by a nontrivial surgery on a hyperbolic knot in
$S^3$ has order at least $18$.
\end{conjecture}

The Cyclic Surgery Theorem implies that every lens space surgery slope
on a hyperbolic knot is an integer \cite[Corollary~1]{CGLS}. Berge
introduced doubly primitive knots and conjectured that every knot in $S^3$
admitting a lens space surgery is doubly primitive \cite{Berge}.
Floer-theoretic work gives strong additional restrictions: a knot admitting
a positive lens space surgery is an L-space knot, hence is fibered, and
its Alexander polynomial is highly constrained \cite{KMOS,NiFibered,OS}.
Greene proved that if positive integer $p$-surgery on a knot $K\subset S^3$
is a lens space, then there is a knot $B$ from Berge's list such that
$S^3_p(B)\cong S^3_p(K)$ and
$\widehat{HFK}(B)\cong\widehat{HFK}(K)$
\cite[Theorem~1.3]{Greene}. Baker showed that any counterexample to
Conjecture~\ref{conj:BL} would, after mirroring if necessary, have to satisfy
\cite[Theorem~1.6 and the end of Section~1.3]{Baker}
\begin{equation}\label{eq:baker}
 S^3_{14}(K)\cong L(14,11),\qquad
 g(K)=4,\qquad \Delta_K(t)=\Delta_{T(3,5)}(t).
\end{equation}
Throughout, $\Delta_K(t)$ denotes the symmetrized Alexander polynomial
normalized by $\Delta_K(1)=1$. With this normalization, the latter polynomial is
\begin{equation}\label{eq:phi15}
 \Delta_{T(3,5)}(t)=t^4-t^3+t-1+t^{-1}-t^{-3}+t^{-4}
 =t^{-4}\Phi_{15}(t).
\end{equation}
We exclude this remaining possibility by means of the following spectral
obstruction.

\begin{theorem}\label{thm:spectral}
Let $K\subset S^3$ be a hyperbolic knot of genus $g$. Suppose that $K$
admits a lens space surgery with slope $\pm(4g-2)$. Then $\Delta_K(t)$
has a real root $\rho$ with $|\rho|>1$.
\end{theorem}

The value $4g-2$ actually occurs: the hyperbolic pretzel knot
$P(-2,3,7)$ has genus $5$ and admits a lens space surgery with slope
$18=4g-2$ \cite{Baker,FintushelStern}. Baker proved that a genus $g$ knot
in $S^3$ admitting a lens space surgery of order at least $4g-1$ is a knot
in Berge's list \cite[Theorem~1.2]{Baker}. Thus $4g-2$ is the first order
below Baker's threshold. Our spectral obstruction yields the following
consequence for knots with torus-knot Alexander polynomial.

\begin{corollary}\label{cor:torus-polynomial}
Let $r,s>1$ be coprime integers, and let $K\subset S^3$ be a hyperbolic knot
such that $\Delta_K(t)=\Delta_{T(r,s)}(t)$. Then no integer slope $p$ with
$|p|=2(r-1)(s-1)-2$ is a lens space surgery slope for $K$.
\end{corollary}

This suggests that the spectral obstruction may be useful more broadly in
studying hyperbolic fibered knots having the same Alexander polynomial as a
torus knot. Combining Corollary~\ref{cor:torus-polynomial} with Baker's
result gives the following.

\begin{corollary}\label{cor:BL}
The Bleiler--Litherland conjecture holds.
\end{corollary}

The spectral obstruction also has consequences for characterizing slopes.
Recall that a slope $p/q$ is characterizing for a knot $J$ if, for every
knot $K\subset S^3$, an orientation-preserving homeomorphism
$S^3_{p/q}(K)\cong S^3_{p/q}(J)$ implies that $K$ is isotopic to $J$.

\begin{corollary}\label{cor:characterizing}
The slopes $11$, $14$, and $22$ are characterizing for $T(3,4)$,
$T(3,5)$, and $T(3,7)$, respectively.
\end{corollary}

The paper is organized as follows. Section~2 collects the background needed
later, including L-space knots and homological monodromy, pseudo-Anosov
dynamics, stable laminations under Dehn filling, and lens space surgeries on
torus and satellite knots. In Section~3 we isolate the mod-$2$ orientability
argument needed to detect the pseudo-Anosov dilatation on homology.
Section~4 combines these ingredients to prove the spectral obstruction and
the Bleiler--Litherland conjecture. Section~5 then applies the obstruction
to characterizing slopes for torus knots.

\section{Preliminaries}

We recall the background used in the proof of the main theorem and in the
applications.

\subsection{L-space knots and monodromy}
Recall that an L-space is a rational homology sphere $Y$ satisfying
$\rank\widehat{HF}(Y)=|H_1(Y)|$. A knot $K\subset S^3$ is an L-space
knot if some positive surgery on $K$ is an L-space. Lens spaces are
L-spaces \cite{OS}; hence a knot admitting a positive lens space surgery
is an L-space knot. Ni's fibered-knot detection theorem then gives the
following.

\begin{proposition}\label{prop:fibered}
If a knot $K\subset S^3$ admits a positive lens space surgery, then $K$
is fibered \cite[Corollary~1.3]{NiFibered}.
\end{proposition}

For hyperbolic L-space knots we use the following consequence of Ni's
fixed-point theorem \cite[Theorem~1.1 and the paragraph following it]{NiFixed};
only the absence of interior fixed points will be needed below.

\begin{theorem}[Ni]\label{thm:Ni}
Let $K\subset S^3$ be a hyperbolic L-space knot with fiber $F$. Then the
monodromy is freely isotopic to a pseudo-Anosov homeomorphism
$\psi:F\to F$ having no fixed points in $\Int F$.
\end{theorem}

The Alexander polynomial and the homological monodromy are related by the
following standard formula; see \cite{Milnor}.

\begin{proposition}\label{prop:alexander}
Let $K\subset S^3$ be a fibered knot with fiber $F$ of genus $g$, and let
$h_*:H_1(F;\Z)\longrightarrow H_1(F;\Z)$ be the homological monodromy.
Then
\begin{equation}\label{eq:alexander}
 \det(tI-h_*)=\pm t^g\Delta_K(t).
\end{equation}
In particular, $\det(I-h_*)=\pm1$.
\end{proposition}

\subsection{Measured foliations and pseudo-Anosov homeomorphisms}
We use pseudo-Anosov dynamics in two settings. We first recall two standard
facts about measured foliations on closed surfaces: an orientability
criterion and the action of a pseudo-Anosov homeomorphism on first
homology. We then turn to a surface with one boundary component and
discuss boundary prongs, the capping construction, and the
Euler--Poincar\'e prong formula.

\begin{lemma}\label{lem:local-orientability}
Let $\Fol$ be a measured foliation on a closed oriented surface $S$, let
$\Sigma$ be its singular set, put $U=S\setminus\Sigma$, and let
$T\Fol|_U$ denote the tangent line bundle of $\Fol$ on $U$. Then $\Fol$
is orientable if and only if $w_1(T\Fol|_U)=0$ in $H^1(U;\Ftwo)$.
If $z\in\Sigma$ is a $p$-pronged singularity and $\gamma_z$ is a small
loop around $z$, then
$\langle w_1(T\Fol|_U),[\gamma_z]\rangle\equiv p\pmod{2}$.
Thus the local orientation obstruction vanishes exactly at even-pronged
singularities.
\end{lemma}

\begin{proof}
The first assertion is the standard orientability criterion for a real line
bundle. Around a $p$-pronged singularity, transporting an orientation once
around a small loop reverses it precisely when $p$ is odd, which gives the
stated evaluation of $w_1(T\Fol|_U)$.
\end{proof}

Orientability also detects the pseudo-Anosov dilatation on first homology.
We use the following standard fact; see \cite[Theorem~2.2]{LanneauThiffeault}.

\begin{proposition}\label{prop:dilatation}
Let $\phi:S\to S$ be a pseudo-Anosov homeomorphism of a closed oriented
surface with dilatation $\lambda>1$. If one of its invariant measured
foliations is orientable, then
$\phi_*:H_1(S;\R)\longrightarrow H_1(S;\R)$ has a real eigenvalue
$\rho=\pm\lambda$.
\end{proposition}

We now turn to surfaces with boundary. Let $F$ be a compact oriented
surface with one boundary component, and let $h:F\to F$ fix $\partial F$
pointwise with pseudo-Anosov mapping class. Choose a pseudo-Anosov
representative $\psi$ freely isotopic to $h$, with stable and unstable
invariant measured foliations $(\Fol^s,\mu^s)$ and $(\Fol^u,\mu^u)$ and
dilatation $\lambda>1$. A boundary prong is a prong of an invariant
foliation ending at a point of $\partial F$; see the left-hand side of
Figure~\ref{fig:capping}. The stable and unstable foliations have the same
singularities and the same boundary-prong number $n\geq1$
\cite[Section~2]{BaldwinHuSivek}.

We now cap off the boundary. For $n\geq2$, let
$q:F\to\Fhat=F/\partial F$ collapse $\partial F$ to a point $v$.
As illustrated in Figure~\ref{fig:capping}, the $n$ boundary prongs extend
across $v$. Thus $v$ is an $n$-pronged singularity if $n>2$ and a regular
point if $n=2$.

\begin{lemma}\label{lem:capping}
Assume that $\psi$ has boundary-prong number $n\geq2$. Then $\psi$
induces a pseudo-Anosov homeomorphism $\psihat:\Fhat\to\Fhat$ with the
same dilatation $\lambda$. The invariant measured foliations extend across
$v$ as described above.
\end{lemma}

\begin{proof}
Since $\psi$ preserves $\partial F$ setwise, it descends to a homeomorphism
$\psihat:\Fhat\to\Fhat$ fixing $v=q(\partial F)$. The invariant measured
foliations extend across $v$ through the standard $n$-prong local model;
see \cite[Section~2]{BaldwinHuSivek}. Their transverse measures extend with
the same scaling factors $\lambda^{-1}$ and $\lambda$. Hence $\psihat$ is
pseudo-Anosov with dilatation $\lambda$.
\end{proof}

\begin{figure}[htbp]
 \centering
 \includegraphics[width=0.96\textwidth]{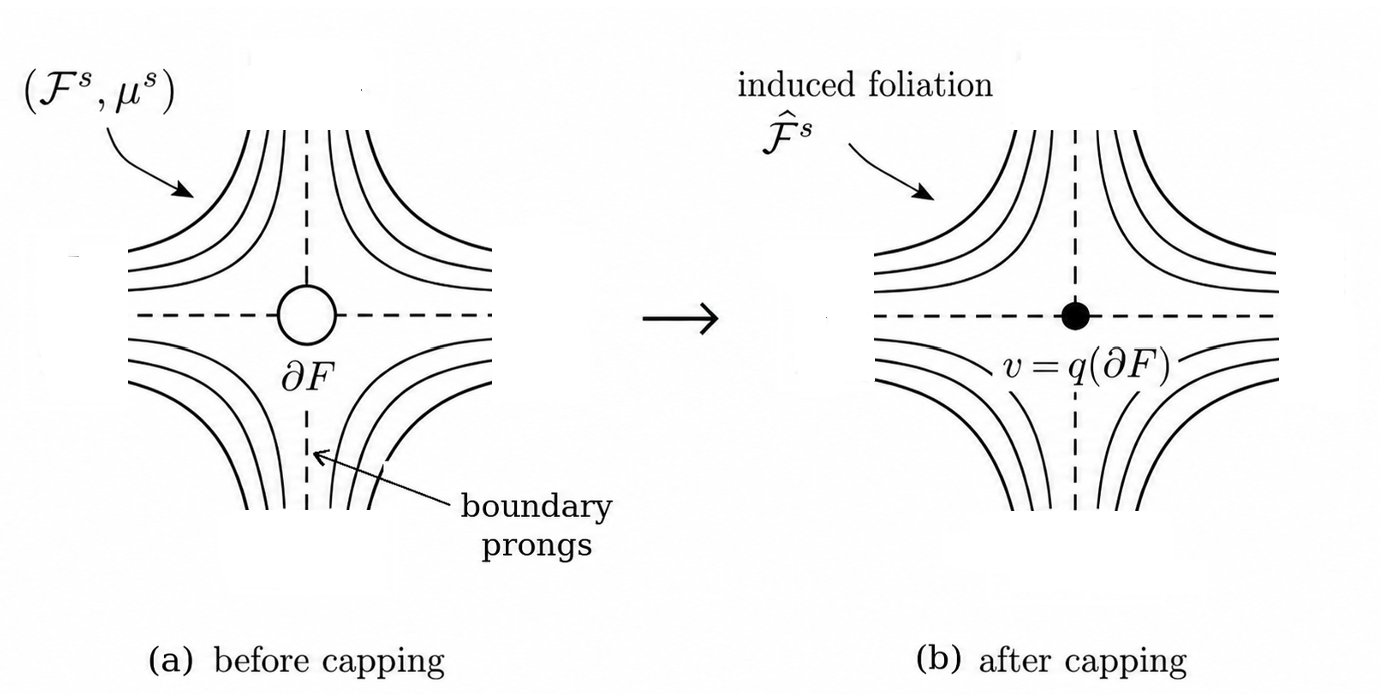}
 \caption{Boundary prongs and the capping construction, shown for $n=4$.
 On the left, the four dashed arcs indicate the boundary prongs ending on
 $\partial F$. After collapsing $\partial F$ to $v$, they extend across
 the capped point.}
 \label{fig:capping}
\end{figure}

The Euler--Poincar\'e formula on the capped surface gives the following
prong relation.

\begin{lemma}\label{lem:euler}
Assume that $\Fol^s$ has boundary-prong number $n\geq2$ and interior
singularities with prong numbers $p_1,\ldots,p_N$. Then
\begin{equation}\label{eq:prongs}
 (n-2)+\sum_{i=1}^N(p_i-2)=4g(F)-4.
\end{equation}
\end{lemma}

\begin{proof}
Apply the standard Euler--Poincar\'e prong formula to the extended
foliation on the closed surface $\Fhat$. The capped point contributes
$n-2$, with contribution $0$ when $n=2$.
\end{proof}

\subsection{Stable laminations and Dehn filling}
Let $K\subset S^3$ be a hyperbolic fibered knot with exterior $X_K$.
Following \cite[Section~5]{GabaiOertel}, let
$\lambda^s\subset\Int X_K$ be the stable lamination transverse to the
fibration; it determines a slope on $\partial X_K$, called its degeneracy
slope. Let $h$ be the monodromy of $K$, and let $c(h)$ denote its
fractional Dehn twist coefficient along $\partial F$; see
\cite[Section~2.1]{NiExceptional}. If a pseudo-Anosov representative of
$h$ has boundary-prong number $n$, then $c(h)=k/n$ for some integer $k$;
see \cite[Section~2]{BaldwinHuSivek}. For slopes $a/b$ and $p/q$ on
$\partial X_K$, let $\Delta(a/b,p/q)=|aq-bp|$ denote their geometric
intersection number. The following proposition records the consequences of
these results that we will need later.

\begin{proposition}\label{prop:degeneracy}
Let $K\subset S^3$ be a hyperbolic fibered knot of genus $g$ admitting a
positive lens space surgery. Let $h$ be its monodromy, let $\psi$ be a
pseudo-Anosov representative freely isotopic to $h$, and let $n$ be its
boundary-prong number. Then:
\begin{enumerate}[label=(\arabic*),leftmargin=*,itemsep=2pt,topsep=4pt]
 \item the stable lamination has degeneracy slope $d/1$ for some integer
 $2\leq d\leq4g-2$;
 \item $n=kd$ for some positive integer $k$;
 \item if a slope $p/q$ satisfies $\Delta(d/1,p/q)\geq2$, then
 $S^3_{p/q}(K)$ is not a lens space.
\end{enumerate}
\end{proposition}

\begin{proof}
For (1), Gabai's theorem \cite[Theorem~8.8]{Gabai}, in the form recorded
by Ni \cite[Theorem~2.4]{NiExceptional}, shows that the degeneracy slope
of $\lambda^s$ on $\partial X_K$ is either the meridional slope $1/0$
or a slope $d/1$ with $2\leq|d|\leq4g-2$. Since $K$ admits a positive
lens space surgery, Ozsv\'ath and Szab\'o's constraints on knot Floer
homology give $\tau(K)=g(K)$ \cite{OS}. By
Proposition~\ref{prop:fibered}, $K$ is fibered. We may therefore apply the
implication $(4)\Rightarrow(2)$ in Hedden's Proposition~2.1 \cite{Hedden}:
the equality $\tau(K)=g(K)$ implies that the open book associated to $K$
supports the unique tight contact structure on $S^3$. Hence its monodromy
is right-veering by \cite[Theorem~1.1]{HKM}. Ni's Proposition~2.3
\cite{NiExceptional} then implies that the degeneracy slope is positive,
so it is $d/1$ with $2\leq d\leq4g-2$. This proves (1).

For (2), write $c(h)=k/n$ for some integer $k$, as above. Ni's reciprocal
formula \cite[Section~2.1]{NiExceptional} gives $c(h)=1/d$. Since $d>0$
by (1), the equality $1/d=k/n$ yields $n=kd$ and $k>0$.

For (3), the theorem of Gabai and Oertel \cite[Theorem~5.3]{GabaiOertel},
in the form recorded by Ni \cite[Theorem~2.5]{NiExceptional}, shows that
$\lambda^s$ remains essential in $S^3_{p/q}(K)$ whenever
$\Delta(d/1,p/q)\geq2$. A closed $3$-manifold containing an essential
lamination has universal cover $\R^3$ \cite[Theorem~6.1]{GabaiOertel};
hence its fundamental group is infinite, so it cannot be a lens space.
\end{proof}

\subsection{Lens space surgeries of torus and satellite knots}
Moser's classification of lens space surgeries on torus knots takes the
following form in our surgery convention.

\begin{theorem}[Moser]\label{thm:Moser}
Let $r>s>1$ be coprime integers, and let $p>0$. Then $p$-surgery on
$T(r,s)$ is a lens space if and only if $p=rs\pm1$. In either case
$S^3_p(T(r,s))\cong L(p,s^2)$ orientation-preservingly. No positive
integer surgery on the mirror of $T(r,s)$ yields a lens space.
\end{theorem}

\begin{remark}\label{rmk:Moser-convention}
Moser uses a different slope convention. Translating her convention to ours
gives Theorem~\ref{thm:Moser}; see \cite[Propositions~3.1--3.2]{Moser}.
\end{remark}

For satellite knots, Zhang's Main Theorem yields the following
classification of lens space surgeries \cite{Zhang}. Related
classifications were obtained independently by Bleiler--Litherland, Wang,
and Wu \cite{BleilerLitherland,Wang,Wu}.

\begin{theorem}\label{thm:satellite}
Let $K\subset S^3$ be a satellite knot and let $p>0$. If $p$-surgery on
$K$ is a lens space, then there are coprime integers $a,b>1$ such that
$p=4ab\pm1$, and $S^3_p(K)$ and $L(p,4b^2)$ are orientation-preservingly
homeomorphic. In particular, $p$ is odd and $p\geq23$.
\end{theorem}

\section{A mod-2 orientability criterion}

Even-pronged singularities have no local orientation obstruction, but a
global obstruction may remain. The following proposition gives a sufficient
homological condition for global orientability.

\begin{proposition}\label{prop:mod2}
Let $F$ be a compact oriented surface with one boundary component. Let
$h:F\to F$ fix $\partial F$ pointwise, and suppose that its mapping
class is pseudo-Anosov. Let $\psi:F\to F$ be a pseudo-Anosov
homeomorphism freely isotopic to $h$. Assume that every interior singularity
of the invariant foliations has an even number of prongs and that the
boundary-prong number $n$ is even. Let
$h_*:H_1(F;\Z)\longrightarrow H_1(F;\Z)$ be the induced map. If
$\det(I-h_*)$ is odd, then the invariant foliations extend to orientable
foliations on the capped surface $\Fhat$.
\end{proposition}

\begin{proof}
Let $h_{*,2}$ denote the induced map on $H_1(F;\Ftwo)$. Since
$\det(I-h_*)$ is odd, $\det(I-h_{*,2})=1$ in $\Ftwo$, and hence
$I-h_{*,2}$ is invertible. From now on, all homology and cohomology groups
in the proof have coefficients in $\Ftwo$.

Since $n$ is positive and even, Lemma~\ref{lem:capping} gives a
pseudo-Anosov homeomorphism $\psihat:\Fhat\to\Fhat$. Let $\Sigmahat$
be the singular set of the capped stable foliation $\Folhat^s$, put
$U=\Fhat\setminus\Sigmahat$, let $i:U\hookrightarrow\Fhat$ be the
inclusion, and let $q:F\to\Fhat$ collapse $\partial F$ to the capped
point. The diagram
\[
\begin{tikzcd}[row sep=2.9em,column sep=3.5em]
 U \arrow[r,"i"] \arrow[d,"\psihat|_U"']
 & \Fhat \arrow[d,"\psihat"]
 & F \arrow[l,"q"'] \arrow[d,"\psi"] \\
 U \arrow[r,"i"'] & \Fhat & F \arrow[l,"q"]
\end{tikzcd}
\]
commutes. Passing to homology gives
\[
\begin{tikzcd}[row sep=2.9em,column sep=3.5em]
 H_1(U) \arrow[r,"i_*"] \arrow[d,"(\psihat|_U)_*"']
 & H_1(\Fhat) \arrow[d,"\psihat_*"]
 & H_1(F) \arrow[l,"q_*"',"\cong"] \arrow[d,"\psi_*"] \\
 H_1(U) \arrow[r,"i_*"']
 & H_1(\Fhat)
 & H_1(F) \arrow[l,"q_*","\cong"']
\end{tikzcd}
\]
and, dually,
\[
\begin{tikzcd}[row sep=2.9em,column sep=3.5em]
 H^1(U)
 & H^1(\Fhat) \arrow[l,"i^*"'] \arrow[r,"q^*","\cong"']
 & H^1(F) \\
 H^1(U) \arrow[u,"(\psihat|_U)^*"]
 & H^1(\Fhat) \arrow[l,"i^*"] \arrow[r,"q^*"']
                 \arrow[u,"\psihat^*"']
 & H^1(F) \arrow[u,"\psi^*"']
\end{tikzcd}
\]
Here $q_*$ is an isomorphism because $F$ has one boundary component;
hence so is $q^*$.

Choose pairwise disjoint closed disks $D_z$ about the points
$z\in\Sigmahat$ and put
$U_0=\Fhat\setminus\bigcup_{z\in\Sigmahat}\Int D_z$.
The inclusion $U_0\hookrightarrow U$ is a homotopy equivalence. The long
exact sequence of the pair $(\Fhat,U_0)$ gives an exact sequence
\[
 \bigoplus_{z\in\Sigmahat}\Ftwo[\gamma_z]
 \longrightarrow H_1(U;\Ftwo)
 \xrightarrow{i_*}H_1(\Fhat;\Ftwo)\longrightarrow0,
\]
where $\gamma_z=\partial D_z$. Dualizing gives
\begin{equation}\label{eq:dual}
 0\longrightarrow H^1(\Fhat;\Ftwo)
 \xrightarrow{i^*}H^1(U;\Ftwo)
 \longrightarrow\Hom\left(
 \bigoplus_{z\in\Sigmahat}\Ftwo[\gamma_z],\Ftwo\right).
\end{equation}
In particular, $i^*$ is injective.

Let $T\Folhat^s|_U$ denote the tangent line bundle of $\Folhat^s$ on
$U$. Every singularity of $\Folhat^s$ is even-pronged: the interior
singularities are even-pronged by hypothesis, while the capped point is
$n$-pronged when $n>2$ and regular when $n=2$. Hence
Lemma~\ref{lem:local-orientability} gives
$\langle w_1(T\Folhat^s|_U),[\gamma_z]\rangle=0$ for every
$z\in\Sigmahat$. By exactness of \eqref{eq:dual}, there is therefore a
unique class $w\in H^1(\Fhat;\Ftwo)$ such that
$i^*w=w_1(T\Folhat^s|_U)$. Since $\Folhat^s$ is $\psihat$-invariant,
naturality of the first Stiefel--Whitney class gives
$(\psihat|_U)^*w_1(T\Folhat^s|_U)=w_1(T\Folhat^s|_U)$.
Thus $i^*(\psihat^*w-w)=(\psihat|_U)^*i^*w-i^*w=0$.
Since $i^*$ is injective, $\psihat^*w=w$.

Finally, since $h$ and $\psi$ are freely isotopic, they induce the same
map on $H_1(F;\Ftwo)$. As $q_*$ is an isomorphism, the right-hand
square of the homology diagram gives
$\psihat_*=q_*h_{*,2}q_*^{-1}$. Hence $I-\psihat_*$ is invertible
because $I-h_{*,2}$ is invertible, and therefore so is $I-\psihat^*$.
Since $\psihat^*w=w$, we obtain $w=0$. Thus
$w_1(T\Folhat^s|_U)=i^*w=0$, so the stable foliation is orientable.
The same argument applies to the unstable foliation.
\end{proof}

\section{The spectral obstruction}

We now prove the spectral obstruction. The argument has two steps.
First, the persistence theorem for the stable lamination forces the
degeneracy slope to lie within distance one of the lens space surgery
slope. The prong formula and the interior fixed-point conclusion of Ni's
theorem then rule out the only remaining alternative to $d=4g-2$. Once
$d=4g-2$, the prong formula
leaves no interior singularities, and the mod-$2$ criterion makes the
invariant foliations orientable after capping. The pseudo-Anosov dilatation
can then be detected on first homology and hence by the Alexander polynomial.

\begin{proof}[Proof of Theorem~\ref{thm:spectral}]
Let $p=\pm(4g-2)$ be the surgery slope. Since
$S^3_{-p}(\overline{K})\cong -S^3_p(K)$, replacing $K$ by its mirror if
necessary allows us to assume $p=4g-2>0$. Mirroring preserves
hyperbolicity, genus, and the Alexander polynomial. By
Proposition~\ref{prop:fibered}, $K$ is fibered. Since $K$ is hyperbolic,
its monodromy $h$ is pseudo-Anosov. Let $\psi$ be the pseudo-Anosov
representative given by Theorem~\ref{thm:Ni}, so $\psi$ has no fixed points
in $\Int F$. Denote by $n$ its boundary-prong number and by
$p_1,\ldots,p_N$ the prong numbers of its interior singularities.

By Proposition~\ref{prop:degeneracy}(1)--(2), the stable lamination has
degeneracy slope $d/1$ with $2\leq d\leq4g-2$, and $n=kd$ for some
positive integer $k$. In particular, $n\geq2$. Hence Lemma~\ref{lem:euler}
applies and gives
\begin{equation}\label{eq:prongs-proof}
 (n-2)+\sum_{i=1}^N(p_i-2)=4g-4.
\end{equation}
Every interior singularity has at least three prongs, so each term $p_i-2$
is positive and hence $n\leq4g-2$.

Since the surgery slope is $p/1$, we have
$\Delta(d/1,p/1)=|d-p|$. If $|d-p|\geq2$,
Proposition~\ref{prop:degeneracy}(3) says that the stable lamination
remains essential after $p$-surgery, which is impossible because $S^3_p(K)$
is a lens space. Hence $|d-p|\leq1$. Since $p=4g-2$ and $d\leq4g-2$,
we have $d\in\{4g-3,4g-2\}$.

We now show that $d=4g-2$, treating the cases $g=1$ and $g\geq2$
separately.

\medskip\noindent\textbf{Case 1:} $g=1$.
Then $2\leq d\leq4g-2=2$, so $d=2=4g-2$.

\medskip\noindent\textbf{Case 2:} $g\geq2$.
Suppose, for contradiction, that $d=4g-3$. Since
$n=k(4g-3)\leq4g-2$ and $k$ is a positive integer, we must have $k=1$.
Hence $n=4g-3$, and \eqref{eq:prongs-proof} gives
$\sum_{i=1}^N(p_i-2)=1$. Each summand is a positive integer, so there is
exactly one interior singularity, and it is three-pronged. A pseudo-Anosov
homeomorphism permutes the singularities of its invariant foliations;
since there is only one interior singularity, it must be fixed by $\psi$.
This contradicts Theorem~\ref{thm:Ni}, because the singularity lies in
$\Int F$. Therefore $d=4g-2$.

Thus in either case $d=4g-2$. Since $n=kd\leq4g-2$, we have $k=1$
and $n=4g-2$. Substituting this into \eqref{eq:prongs-proof} gives
$\sum_{i=1}^N(p_i-2)=0$, so there are no interior singularities.
Moreover, $n=4g-2$ is even.

Hence, after capping $\partial F$, every singularity of the extended
invariant foliations is even-pronged: there are no interior singularities,
and the capped point is $n$-pronged when $n>2$ and regular when $n=2$.
By Proposition~\ref{prop:alexander}, $\det(I-h_*)=\pm1$, which is odd.
All hypotheses of Proposition~\ref{prop:mod2} are therefore satisfied, so
the invariant foliations extend to orientable foliations on the capped
surface $\Fhat$. By Lemma~\ref{lem:capping}, the induced map
$\psihat:\Fhat\to\Fhat$ is pseudo-Anosov with the same dilatation
$\lambda>1$, and Proposition~\ref{prop:dilatation} gives a real eigenvalue
$\rho=\pm\lambda$ of $\psihat_*:H_1(\Fhat;\R)\to H_1(\Fhat;\R)$.
Since $F$ has one boundary component, collapsing $\partial F$ to the
capped point induces an isomorphism
$q_*:H_1(F;\R)\to H_1(\Fhat;\R)$. The relation
$q\circ\psi=\psihat\circ q$ shows that $\psi_*$ and $\psihat_*$ are
conjugate, so $\rho$ is an eigenvalue of $\psi_*$, and hence also of
$h_*$, since $h$ and $\psi$ are freely isotopic. Thus
$\det(\rho I-h_*)=0$, and \eqref{eq:alexander} gives $\Delta_K(\rho)=0$.
Since $|\rho|=\lambda>1$, the theorem follows.
\end{proof}

We now prove the two corollaries stated in the Introduction.

\begin{proof}[Proof of Corollary~\ref{cor:torus-polynomial}]
Suppose that $S^3_p(K)$ is a lens space for an integer $p$ with
$|p|=2(r-1)(s-1)-2$. After mirroring, as in the proof of
Theorem~\ref{thm:spectral}, we may assume $p>0$. By
Proposition~\ref{prop:fibered}, $K$ is fibered, and so is $T(r,s)$.
Hence equality of their Alexander polynomials implies equality of their
genera, so $g(K)=g(T(r,s))=(r-1)(s-1)/2$. Consequently,
$p=2(r-1)(s-1)-2=4g(K)-2$. Theorem~\ref{thm:spectral} therefore
implies that $\Delta_K(t)$ has a real root $\rho$ with $|\rho|>1$.

On the other hand,
\[
 \Delta_{T(r,s)}(t)=t^{-\frac{(r-1)(s-1)}{2}}
 \frac{(t^{rs}-1)(t-1)}{(t^r-1)(t^s-1)}.
\]
Hence every zero of $\Delta_{T(r,s)}(t)$ is a root of unity. This is a
contradiction.
\end{proof}

\begin{proof}[Proof of Corollary~\ref{cor:BL}]
By Baker's result, any counterexample to the Bleiler--Litherland conjecture
would, after mirroring if necessary, satisfy \eqref{eq:baker}. In
particular, $g(K)=4$, $\Delta_K(t)=\Delta_{T(3,5)}(t)$, and
$p=14=2(3-1)(5-1)-2=4g(K)-2$.
Corollary~\ref{cor:torus-polynomial} excludes this case. Hence no
counterexample exists.
\end{proof}

\section{Characterizing slopes}

\begin{proof}[Proof of Corollary~\ref{cor:characterizing}]
For the three torus knots under consideration, the relevant data are
summarized below:
\[
\begin{array}{c|c|c|c}
 J & g(J) & p & S^3_p(J)\\ \hline
 T(3,4) & 3 & 11=4g(J)-1 & L(11,9)\\
 T(3,5) & 4 & 14=4g(J)-2 & L(14,9)\\
 T(3,7) & 6 & 22=4g(J)-2 & L(22,9)
\end{array}
\]
Suppose that $S^3_p(K)\cong S^3_p(J)$ orientation-preservingly for one
of the three rows above. By the oriented classification of lens spaces
\cite{Brody}, none of the target lens spaces is orientation-preservingly
homeomorphic to $L(p,1)$. Since $p$-surgery on the unknot is $L(p,1)$,
the knot $K$ is nontrivial. Since positive $p$-surgery on $K$ is an
L-space, $K$ is an L-space knot and hence is prime by Krcatovich's
theorem \cite[Theorem~1.2]{Krcatovich}. Thus $K$ is a torus knot, a
satellite knot, or a hyperbolic knot. We treat the three slopes separately.

\medskip\noindent\textbf{Case 1:} $p=11$.
The hyperbolic case is excluded by Corollary~\ref{cor:BL}, and the
satellite case by Theorem~\ref{thm:satellite}. Thus $K$ is a torus knot.
Moser's theorem gives $ab=10$ or $12$, so $K$ is either $T(5,2)$ or
$T(4,3)$. Since $S^3_{11}(T(5,2))\cong L(11,4)$, which is not
orientation-preservingly homeomorphic to the target $L(11,9)$ because
$4\not\equiv9^{\pm1}\pmod{11}$, we obtain $K=T(4,3)=T(3,4)$.

\medskip\noindent\textbf{Case 2:} $p=14$.
The hyperbolic case is excluded by Corollary~\ref{cor:BL}, and the
satellite case by Theorem~\ref{thm:satellite}. Thus $K$ is a torus knot.
Moser's theorem gives $ab=13$ or $15$. Since $a,b>1$, the case $ab=13$
is impossible, while $ab=15$ gives $K=T(5,3)=T(3,5)$.

\medskip\noindent\textbf{Case 3:} $p=22$.
The satellite case is excluded by Theorem~\ref{thm:satellite}. If $K$
is a torus knot, Moser's theorem gives $ab=21$ or $23$. Since $a,b>1$,
the case $ab=23$ is impossible, while $ab=21$ gives
$K=T(7,3)=T(3,7)$. It remains to exclude the hyperbolic case. Greene's
realization theorem \cite[Theorem~1.3]{Greene} provides a knot $B$ from
Berge's list such that $S^3_{22}(B)\cong S^3_{22}(K)\cong L(22,9)$
and $\widehat{HFK}(B)\cong\widehat{HFK}(K)$. We may therefore use
Ichihara--Saito's result \cite{IchiharaSaito}. In Ichihara--Saito's
convention, $q=5$ or $9$ \cite{Brody}, and
\cite[Table~1]{IchiharaSaito} gives $B=T(7,3)$ in either case.
The knot Floer homology isomorphism gives $g(K)=g(T(3,7))=6$ and
$\Delta_K(t)=\Delta_{T(3,7)}(t)=t^{-6}\Phi_{21}(t)$.
Since $22=4g(K)-2$, Theorem~\ref{thm:spectral} forces a real root of
$\Delta_K$ outside the unit circle, whereas every root of $\Phi_{21}$
lies on the unit circle. This contradiction excludes the hyperbolic case
and completes the proof.
\end{proof}

\section*{Acknowledgments}
The author thanks Steven Sivek for pointing out an issue with the argument
in an earlier version of Lemma~2.7 and for helpful correspondence.

\section*{AI-use disclosure}
The author used AI tools to polish the language and improve sentence fluency.
The author takes full responsibility for the entire content of this manuscript.

\end{document}